\documentclass{article}
\usepackage{amsmath}
\usepackage{amssymb}
\usepackage{amsthm}
\usepackage{bbm,mathtools}
\usepackage{accents}
\usepackage{graphicx}
\usepackage{stmaryrd}

\newcommand\cev[1]{\overset{{}_{\shortleftarrow}}{#1}}

\newcommand\undervec[1]{\underaccent{\vec}{#1}}
\newcommand{\op}[1]{\prescript{o}{}{#1}}
\newcommand{\pp}[1]{\prescript{p}{}{#1}}

\newcommand{\graph}[1]{(\mathrm {gph}\ #1)}

\newcommand{\midb}{\;\middle|\;}
\newcommand{\one}{\mathbbm 1}

\newcommand{\cadlag}{\mbox{c\`adl\`ag}} 
\newcommand{\cad}{\mbox{c\`ad}}
\newcommand{\lag}{\mbox{l\`ag}}

\newcommand{\caglad}{\mbox{c\`agl\`ad}} 
\newcommand{\cag}{\mbox{c\`ag}}
\newcommand{\lad}{\mbox{l\`ad}}

\def\reals{\mathbb{R}}
\def\rationals{\mathbb{Q}}
\def\naturals{\mathbb{N}}
\def\uball{\mathbb{B}}
\def\ereals{\overline{\mathbb{R}}}

\def\cone{\mathop{\rm cone}}

\def\interior{\mathop{\rm int}\nolimits}

\def\comp{\raise 1pt \hbox{$\scriptstyle\circ$}}

\def\argmin{\mathop{\rm argmin}\limits}

\def\dom{\mathop{\rm dom}\nolimits}

\def\rge{\mathop{\rm rge}}

\def\naturals{\mathbb{N}}
\def\upto{{\raise 1pt \hbox{$\scriptstyle \,\nearrow\,$}}}
\def\downto{{\raise 1pt \hbox{$\scriptstyle \,\searrow\,$}}}

\def\cl{\mathop{\rm cl}\nolimits}
\def\co{\mathop{\rm co}}

\def\epi{\mathop{\rm epi}}

\def\tos{\rightrightarrows}
\def\FF{(\F_t)_{t\ge 0}}

\def\A{{\cal A}}
\def\B{{\cal B}}
\def\C{{\cal C}}

\def\D{{\cal D}}

\def\F{{\cal F}}

\def\LL{{\mathbb L}}
\def\M{{\cal M}}

\def\O{{\cal O}}

\def\T{{\cal T}}

\def\X{{\cal X}}

\def\nn{\mathbb{N}}
\newcommand{\rpr}{\mathbb{P}} 

\newtheorem{theorem}{Theorem}
\newtheorem{lemma}[theorem]{Lemma}
\newtheorem{corollary}[theorem]{Corollary}

\newtheorem{example}{Example}

\theoremstyle{definition}

\newtheorem{assumption}{Assumption}
\theoremstyle{empty}

\begin{document}
\title{Convex integral functionals of $\cadlag$ processes}
\author{Ari-Pekka Perkki\"o\thanks{Department of Mathematics, Ludwig Maximilians Universit\"at M\"unchen, Theresienstr. 39, 80333 M\"unchen, Germany.}
\and
Erick Trevi\~no-Aguilar\thanks{Department of Economics and Finance, University of Guanajuato, Lascur\'ain de Retana 5,  36000, Guanajuato, M\'exico.}}

\maketitle

\begin{abstract}
This article characterizes conjugates and subdifferentials of convex integral functionals over linear spaces of $\cadlag$ stochastic processes. The approach is based on new measurability results on the Skorokhod space and  new interchange rules of integral functionals that are developed in the article. The main results provide a general approach to apply convex duality in a variety of optimization problems ranging from optimal stopping to singular stochastic control and mathematical finance.
\end{abstract}


\section{Introduction}

We fix a complete stochastic base $(\Omega,\F,\FF,P)$ satisfying the usual hypotheses with a terminal time $T>0$ which we allow to be $+\infty$. For a space $\D$ of adapted processes with $\cadlag$  paths  (the french abbreviation for right continuous  with left limits), we study conjugates of convex integral functionals
\begin{equation}\label{eq:F}
F(y)= E\int_{[0,T]} h(y) d \mu+\delta_{\D(S)}(y)
\end{equation}
for a convex normal integrand $h:\Omega \times [0,T] \times \reals^d\to\ereals$, a nonnegative random measure $\mu$ and hard constraints $\D(S)$. Here $\D(S)$ is a subset of $\D$ consisting of almost sure selections of the image closure $S$ of $\dom h_t(\omega)=\{x\mid h_t(x,\omega)<+\infty\}$ and  $\delta_{\D(S)}(y)$ takes the value zero if $y$ belongs to $\D(S)$ and $+\infty$ otherwise. 

For a large class of Banach spaces of adapted $\cadlag$ processes, the dual space can be identified with pairs of random measures under the bilinear form
\[
\langle y,(u,\tilde u)\rangle=E\left[\int ydu + \int y_-d\tilde u\right],
\]
where $y_-$ is the left continuous version of $y$; see \cite{ara14,dm82,pp18b}. This leads us to analyze a larger class of functionals
\begin{equation}\label{eq:hatF}
\hat F(y)= E\left[ \int_{[0,T]} h(y) d \mu+ \int_{[0,T]} \tilde h(y_-) d \tilde\mu\right]+\delta_{\D(S)}(y)+\delta_{\D_l(\tilde S)}(y_-),
\end{equation}
where $\tilde h$ is another convex normal integrand, $\tilde\mu$  another non negative random measure, $\D_l$ a space of $\caglad$ (the french abbreviation for left continuous  with right limits) stochastic processes and $\tilde S=\cl \dom(\tilde h)$.

The main results of the paper, Theorems~\ref{thm:cifcp}  and  \ref{thm:cifcpcor}, characterize the conjugates and subdifferentials of $F$ and $\hat F$. Our motivation arises from various optimization problems in stochastic control and mathematical finance. Applications to stochastic singular control are already presented in \cite{pp18d} which dealt with integral functionals on regular processes (processes that are optional projections of continuous processes).  Examples in Section~\ref{sec:applications} illustrate how our results allow to model, e.g.,  bid-ask spreads and currency markets in mathematical finance while an  application of our results to optimal stopping is given in \cite{pp18c}. Applications to partial hedging of American options will be presented in a forthcoming article by the authors. Further applications to finance and singular stochastic control going beyond \cite{pp18d} will be given elsewhere.


In Section~\ref{sec:ni}, we recall basic facts from convex analysis and the theory of integral functionals. Since we work later on with the Skorokhod space where the functionals $F$ and $\hat F$ are not lower semicontinuous, we also present an extension of ``an interchange rule between minimization and integration'' to jointly measurable integrands on a non necessarily topological vector Suslin space. Another new interchange rule is provided in Section~\ref{sec:det} for integral functionals on the space $D$ of \cadlag\ functions. Such results go back to the seminal paper of \cite{Roc68} in decomposable spaces of $\reals^d$-valued measurable functions. Extensions to Suslin-valued functions are studied in  \cite{val75} and to non-necessarily decomposable spaces, e.g., in \cite{bv88,pp18a,roc71}. We build on their results.

Section~\ref{sec:graph} addresses the complication that measurable selection theorems behind the interchange rules require Suslin spaces, which $D$ under the supremum norm is not. However, $D$ becomes a Suslin space under the Skorokhod topology. A drawback is that the hard constraints are no longer closed in this topology. The main result of this section, which gives graph measurability of a set-valued mapping on $D$, and hence that of the hard constraints, has a novel method of the proof since we can not use standard characterizations of measurability based on closed-valuedness.

Before presenting the main results in Section~\ref{sec:main}, we develop, in Section~\ref{sec:ifcp},  the crucial interchange rule on a space of $\cadlag$ processes.  Stochastic settings for interchange rules and convex duality  has been recently developed in   \cite{pp18d} and \cite{pp18a} on spaces of regular processes and processes of bounded variation.  Our interchange rule can be seen as an extension of one in \cite{pp18d} from the class of regular processes  to classes of $\cadlag$ processes. 

One of the main assumptions in our main results is a sort of Michael representation (see \cite{rw98}) of a set-valued mapping $S$ consisting of \cadlag\ selections. We analyze this condition in terms of standard continuity properties of set-valued mappings in Section~\ref{sec:SVA}. This section is of independent interest in set-valued analysis. In Section~\ref{sec:applications}, we demonstrate how these results lead to straight-forward applications of the main theorems.

\section{Convex conjugates and normal integrands }\label{sec:ni}
In this section we recall some fundamentals from convex analysis and the theory of normal integrands. We also give a minor extension of an interchange rule for integral functionals on decomposable spaces that we will need in the sequel. 

When $X$ is in separating duality with another linear space $V$, the {\em conjugate} of an extended real-valued convex function $g$ on $X$ is the extended real-valued function $g^*$ on $V$ defined by
\[
g^*(v) = \sup_{x\in X}\{\langle x,v\rangle - g(x)\}.
\]
A vector $v\in V$ is a {\em subgradient} of $g$ at $x$ if
\[
g(x')\ge g(x) + \langle x'-x,v\rangle\quad\forall x'\in X.
\]
The {\em subdifferential} $\partial g(x)$ is the set of all subgradients of $g$ at $x$. We often use the property that $v\in\partial g(x)$ if and only if
\[
g(x)+g^*(v)=\langle x,v \rangle.
\]
The {\em recession function} of a closed proper convex function $g$ is defined by
\[
g^\infty(x)=\sup_{\alpha>0}\frac{g(\bar x+\alpha x)-g(\bar x)}{\alpha},
\]
where the supremum is independent of the choice of $\bar x\in\dom g=\{x \in X \mid g(x)<\infty \}$; see \cite[Corollary~3C]{roc66}. By \cite[Corollary 3D]{roc66},
\begin{equation}\label{eq:domrec}
\sigma_{\dom g^*}=g^\infty,
\end{equation}
where $\sigma_C:=\delta_C^*$ is the {\em support function} of $C$.  

Let $X$ be a topological space equipped with its Borel $\sigma $-algebra $\B(X)$ and $(\Xi,\A,m)$ a measure space.  A set-valued mapping $S:\Xi \tos X$ is {\em measurable} if the inverse image  $S^{-1}(O):=\{\xi \in \Xi \mid S(\xi) \cap O  \neq \emptyset\}$ of every open set $O$ is $\A$-measurable. A function $ h : \Xi \times X \to \ereals$ is a \textit{normal integrand} on $X$ if its epigraphical mapping
\[
\epi h (\xi) := \{(x,\alpha) \in   X \times \reals \mid  h (\xi,x) \leq \alpha\}
\]
is closed-valued and measurable. When this mapping is also convex-valued, $ h $ is  a \textit{convex normal integrand}. A general treatment of normal integrands on $\reals^d$ can be found from \cite[Chapter 14]{rw98} while integrands on a Suslin space are studied in \cite{cv77}. In particular, a normal integrand $h$ is jointly measurable so that the {\em integral functional}
\[
\int h(w(\xi),\xi)dm
\]
is well-defined for any measurable $w:\Xi\to X$. Throughout the article, an integral is defined as $+\infty$ unless the positive part is integrable.

When $\A$ is  $m$-complete, every jointly measurable $ h : \Xi \times X \to \ereals$, that is lower semicontinuous in the second argument almost everywhere, is a normal integrand whenever $X$ and $X^*$ are Suslin locally convex spaces; see e.g., \cite[Lemma VII-1]{cv77}  where this is actually taken as the definition of a normal integrand. Later on, we will work with non-complete $\sigma$-algebras (the predictable and optional $\sigma$-algebras) 
so we use the definition given in terms of the epigraphical mapping.


For a normal integrand $h$, $h^\infty$ is the integrand defined $\xi$-wise as the recession function of $h(\xi,\cdot)$.  The {\em conjugate integrand} $h^*$ is defined $\xi$-wise as the conjugate of $h(\xi,\cdot)$. When $X$ and $X^*$ are Suslin and locally convex spaces,  $h^*$ is a convex normal integrand due to \cite[Corollary VII-2]{cv77}. For our purposes, $h^\infty$ is a convex normal integrand when $X=\reals^d$, see \cite[Exercise 14.54]{rw98}. Given a measurable function $w$, we denote by $\partial h(w)$ the set-valued mapping $\xi\mapsto\partial h(w(\xi),\xi)$.

Interchange rules for integral functionals on decomposable spaces go back to \cite{Roc68} for normal integrands on $\reals^d$. Theorem~\ref{thm:ifs} below is an extension of the main theorem in \cite{val75} that was formulated for a locally convex Suslin topological vector space $X$. However, our formulation is closer to  \cite[Theorem~14.60]{rw98} stated for $X=\reals^d$. We need this extension later on since the space of $\cadlag$ functions equipped with the Skorokhod topology is not a topological vector space. The proof is almost identical, but we give it here for completeness. 

Recall that $\X\subset L^0(\Xi;X)$ is {\em decomposable} if, for every $x\in\X$, $\one_A \bar x+ \one_{A^C} x \in \X$ whenever $A\in\A$ and $\bar x\in L^0(\Xi;X)$ with $m(A)<\infty$ and $\cl(\rge(\bar x))$ is compact.  
\begin{theorem}\label{thm:ifs}
Assume that $(\Xi,\A,m)$ is a complete $\sigma$-finite measure space, and $X$ is a Suslin space.  Let $h:\Xi\times X \to \ereals$ be $\A \otimes \B(X)$-measurable. Then, for $\X \subset L^0(\Xi;X)$  decomposable 
\[
\inf_{x\in\X} \int h(x)dm=\int \inf_{x\in X} h(x)dm
\]
as soon as the left side is less than $+\infty$.
\end{theorem}
\begin{proof}
We note first that
\[
p(\xi):=\inf_{x \in X} h(\xi,x)
\]	
is measurable. Indeed, for the projection  $\prod: \Xi \times X \to \Xi$ and any $b\in\reals$, we have $p^{-1}((\infty,b))=\prod h^{-1}((-\infty,b))$, where the right side is measurable due to \cite[III.44-45]{dm78}, since $\A$ is complete.

Let $\alpha>\int p dm$. There exists $\beta\in L^0(\Xi)$ such that $\beta>p $ and $\int \beta dm <  \alpha$. Indeed, by $\sigma$-finiteness, there exists strictly positive $\bar\beta\in L^1(\Xi,m)$ so that, for small enough $\epsilon$, we may choose $\beta:=\epsilon \bar\beta+p$. The set 
\[
A:=\{(\xi,x)\in \Xi \times X\mid h(\xi,x) < \beta(\xi)\}
\]
is $\A \otimes \B(X)$-measurable and $\Pi A =\Xi$, so there exists $x\in L^0(X)$ such that $h(x) < \beta$ by \cite[III.44-45]{dm78}.

By \cite[Theorem 7.4.3]{bog7}, the law of $m\circ x^{-1}$ is Radon on $X$, so there exists a nondecreasing sequence $(\Xi_\nu)$, $\bigcup \Xi_\nu=\Xi$  such that $\cl x(\Xi_\nu)$ are compact and $\int 1_{\Xi_\nu^C} xdm<1/\nu$. Let now $\bar x\in\X$ be such that $\int h(\bar x)dm<\infty$, and define
\[
\bar x^\nu=1_{\Xi_\nu}x+1_{\Xi_\nu^C} \bar x.
\]
 By construction, $\bar x^\nu\in\X$, and $h(x^\nu)\le \max\{h(x),h(\bar x)\}$ for all $\nu$, so, by Fatou's lemma,  $\int h(x^\nu) dm <\alpha$ for $\nu$ large enough. 
\end{proof}

\section{Integral functionals of \cadlag\ functions}\label{sec:det}
This section gives an interchange rule for integral functionals on the space of \cadlag\ functions. Our main results  build on the interchange rule.

The space $D$ of $\reals^d$-valued \cadlag\ functions  on $[0,T]$ is Banach for the norm
\begin{equation*}
\left\| y \right\|_{\infty}=\sup_{t \in [0,T]}\left| y_t\right|.  
\end{equation*}
We denote $y_{0-}:=0$ and note that $\lim_{t \nearrow T} y_t$ exists also in the case $T=\infty$. The  dual of $D$ can be identified with $\hat M:= M\times\tilde M$ under the bilinear form
\[
\langle y,\hat u\rangle := \int ydu+\int y_-d\tilde u,
\]
where $\hat u=(u,\tilde u) \in \hat M$, $M$ is the space of $\reals^d$-valued Radon measures on $[0,T]$, $\tilde M\subset M$ is the space of purely atomic measures, and $y_-$ denotes the left continuous version of $y$; this is a deterministic special case of  \cite[Theorem VII.65 and Remark VII.4 (a)]{dm82}. 

Given  a nonnegative  Radon measure $\mu$ on $[0,T]$ and a convex normal integrand $h:[0,T]\times\reals^d\to\ereals$ on $\reals^d$,  the associated integral functional on the space of measurable $\reals^d$-valued functions is
\[
I_h(y):=\int h(y)d\mu:=\int_{[0,T]}h_t(y_t)d\mu_t.
\]
For the {\em domain mapping} $\dom h_t:=\{y\in\reals^d\mid h_t(y)<\infty\}$, 
\begin{equation*}
S_t :=\{ y\in\reals^d \mid y\in\cl \dom h_t\}
\end{equation*}
is its  image closure and
\begin{equation*}
D(S):=\{y\in D\mid y_t\in S_t\ \forall t\}
\end{equation*}
is the set of {\em \cadlag\ selections} of $S$. 

Theorem~\ref{thm:ifcadlag} below is an interchange rule for integral functionals on the nondecomposable space $D$. The assumptions of the theorem are analogous to those of \cite[Theorem~5]{per18a} that was formulated on continuous functions.
\begin{assumption}\label{ass:clS}
We have
\begin{align}\label{ass:clS1}
S_t&= \cl \{y_t \mid y \in D(S)\}\quad\forall\ t,\\
D(S) &=\cl(\dom I_h\cap D(S)),\nonumber
\end{align}
where the latter closure is with respect to the supremum norm.
\end{assumption}

\begin{theorem}\label{thm:ifcadlag}
Under Assumption \ref{ass:clS},
\[
\inf_{y\in D(S)} I_h(y)  = \int \inf_{y\in\reals^d} h(y)d\mu
\]
as soon as the left side is less than $+\infty$.
\end{theorem}
\begin{proof}
We have $D(S)=\cl(\dom I_h\cap D(S))$ and $D(S)$ is PCU-stable in the sense of \cite{bv88}, so, by \cite[Theorem~1]{bv88},
\[
\inf_{y\in D(S)} I_h(y) = \int \inf_{y\in\Gamma_t} h_t(y)d\mu_t, 
\]
where $\Gamma$ is the essential supremum of $D(S)$, i.e., the smallest (up to a $\mu$-null set) closed-valued mapping for which every $y\in D(S)$ is a selection of $\Gamma$ $\mu$-almost everywhere. We have $S_t\subseteq\Gamma_t$ $\mu$-almost everywhere by \eqref{ass:clS1}, so the infimum over $\Gamma_t$ can be taken instead over all of $\reals^d$.
\end{proof}

The above result has a variant on the space $D_l$ of \caglad\ functions (the french abbreviation of left continuous  with right limits), which we equip likewise with the supremum norm.  Given  another nonnegative  Radon measure $\tilde \mu$ on $[0,T]$ and  convex normal integrand $\tilde h$, we define likewise 
\[
I_{\tilde h}(y):=\int \tilde h(y) d\tilde \mu:=\int  \tilde h_t(y_t)d \tilde \mu_t
\]
on the space of measurable $\reals^d$-valued functions, $\tilde S_t:=\cl \dom \tilde h_t$ and 
\begin{align*}
D_l(\tilde S)&:= \{y\in D_l \mid y_t\in \tilde S_t\ \forall t\}.
\end{align*}
We use the convention $\tilde S_{0-}=\{0\}$.

\begin{assumption}\label{ass:clSp}
We have
\begin{align*}
 \tilde S_t&=\cl \{y_t \mid y \in D_l(\tilde S)\}\quad\forall\ t,\\
 D_l(\tilde S) &=\cl(\dom I_{\tilde h}\cap D_l(\tilde S)),
\end{align*}
where the latter closure is with respect to the supremum norm.
\end{assumption}

\begin{theorem}\label{thm:ifcaglad}
Under Assumption \ref{ass:clSp}
\[
\inf_{y\in D_l(\tilde S)} I_{\tilde h}(y) = \int  \inf_{y\in\reals^d} \tilde h(y)d \tilde \mu
\]
as soon as the left side is less than $+\infty$.
\end{theorem}


Next we state the deterministic special case of our main result, Theorem~\ref{thm:cifcp} below. That is, we give a pointwise expression for the conjugate of 
\[
\hat f(y)=I_h(y)+\delta_{D(S)}(y)+I_{\tilde h}(y_-) + \delta_{D_l(\tilde S)}(y_-)
\]
and for its subdifferentials. This result can be used to extend  duality results in deterministic singular control using the machinery in, e.g., \cite{roc78,pp14}.

It will turn out that the conjugate can be expressed in terms of $J_{h^*}$, where for a normal integrand $h$, the functional $J_h:M\to\ereals$ is defined by
\begin{equation}\label{eq:Jh}
J_h(\theta)=\int h(d\theta^a/d\mu)d\mu+\int h^\infty(d\theta^s/d|\theta^s|)d|\theta^s|.
\end{equation}
Here $\theta^a$ and $\theta^s$ are the absolutely continuous and the singular part, respectively, of $\theta$ with respect to $\mu$ and $|\theta^s|$ is the total variation of $\theta^s$. The formula \eqref{eq:domrec} makes it evident that the properties of the domain mapping $S$ play an important role for the validity of the representation in terms of $J_{h^*}$. For integral functionals on continuous functions,  such conjugate formulas go back to \cite{roc71}; see also \cite{per18a} for necessity and sufficiency. We use the notation $\partial^sh:=\partial\delta_{S}$.

\begin{theorem}

Assume that $\tilde\mu$ is purely atomic and, for every $\epsilon>0$, $ y  \in\dom I_{ h } \cap  D(S)$ and $\tilde y\in \dom I_{\tilde h}\cap D_l(\tilde S)$ and $(u,\tilde u)\in\hat M$, there exists $\check  y  \in D(S)$ and $\hat  y  \in D_l(\tilde S)$ with
\begin{enumerate}
\item  $I_{ h }(\check  y )+\int \check ydu \le I_{ h }( y ) + \int ydu +\epsilon$ and $\check y_-\in\dom I_{\tilde h}\cap D_l(\tilde S)$,
\item  $I_{\tilde h }(\hat  y ) +\int \hat yd\tilde u \le I_{ \tilde h }( \tilde y ) +\int \tilde yd\tilde u+\epsilon$ and $\hat y_+\in\dom I_{h}\cap D(S)$.
\end{enumerate}
Then, under Assumptions \ref{ass:clS} and \ref{ass:clSp},  $\hat f$ is a proper lower semicontinuous convex function on $D$,
\[
\hat f^*(u,\tilde u)=J_{h^*}(u)+J_{\tilde h^*}(\tilde u)
\]
and $(u,\tilde u)\in\partial \hat f(y)$ if and only if
\begin{align*}
du/d\mu &\in \partial h(y)\quad \mu\text{-a.e.},\\
du/d|u^s| &\in \partial^s h(y)\quad |u^s|\text{-a.e.},\\
d\tilde u/d\tilde\mu &\in \partial \tilde h(y_-)\quad \mu\text{-a.e.},\\
d\tilde u/d|\tilde u^s| &\in \partial^s \tilde h(y_-)\quad |\tilde u^s|\text{-a.e.}
\end{align*}
\end{theorem}

\section{Graph measurability of integral functionals}\label{sec:graph}
In this section we endow $\Xi=\Omega  \times [0, T ]$ with the product $\sigma$-algebra $\A:=\F\otimes\B([0,T])$. From now on, we fix a convex normal integrand $ h :\Xi \times \reals^d \to \ereals$. For a measurable closed convex-valued mapping $\Gamma:\Omega \times [0,T] \rightrightarrows\reals^d$ we define $D(\Gamma):\Omega \tos D$  by
\[
D(\Gamma) (\omega):=\{y\in D \mid y_t \in \Gamma_t(\omega) \thinspace \forall\ t \}. 
\]
Similarly, we define, pathwise, $I_h$  for any $\reals^d$-valued measurable function $w$ on $[0,T]$ and $J_h$ for any $u\in M$ as
\begin{align*}
I_h(\omega,w)&:=I_{h(\omega,\cdot)}(w),\\
J_{h^*}(\omega,u) &:=J_{h^*(\omega,\cdot)}(u).
\end{align*}

We will consider \cadlag\ processes as $D$-valued random variables where the measurability is understood with respect to the Borel-$\sigma$-algebra $\B(D)$ generated by the Skorokhod topology $\tau_{S}$.  The reason is, that $\tau_{S}$ is Suslin whereas the topology $\tau_D$ generated by the supremum norm is not (since, e.g., $\tau_D$ is not separable). However, $\B(D)$ coincides with the Borel-$\sigma$-algebra generated by the continuous linear functionals on $\tau_D$; see \cite[Theorem~3]{pes95}. It is also generated by point-evaluations; see \cite[Theorem 12.5]{bil99}. For the definition of  the Skorokhod topology and its basic properties, we refer to  \cite{Jacod2002} and \cite{hwy92}.

\begin{lemma}\label{lem:ni1}
The function $I_h: \Omega\times D\rightarrow\ereals$ is $\F\otimes\B(D)$-measurable.
\end{lemma}
\begin{proof}
By \cite[Theorem 15.12]{hwy92}, the sequential convergence in the Skorohod topology implies pointwise convergence outside a countable set. Thus, by \cite[Theorem~20]{pp18d}, $I_{ h }:\Omega\times D\rightarrow\ereals$ is $\F\otimes\B(D)$-measurable.
\end{proof}

\begin{lemma}\label{lem:jointmeas}
For a random time $\tau$, the mappings $(\omega,y)\mapsto y_{\tau(\omega)}$ and $(\omega,y)\mapsto y_{\tau(\omega)-}$ are $\F\otimes\B(D)$-measurable.
\end{lemma}
\begin{proof}
We assume that $0<\tau<T$ almost surely, the extension to general $\tau$ is evident. Defining $g^\nu:\Omega\times D\to\reals$ by
\[
g^\nu(\omega,y) := \frac{1}{\nu}\int_{\tau(\omega)}^{\tau(\omega)+1/\nu} y_sds,
\]
$g^\nu(\cdot,y)$ is measurable and $g^\nu(\omega,\cdot)$ is continuous. Here the continuity follows from the dominated convergence and the facts that convergence in the Skorokhod topology implies uniform boundedness and pointwise convergence outside a countable set. By \cite[Lemma 4.51]{Aliprantis2005}, $g^\nu$ is jointly measurable, so we get the first claim from $\lim_{\nu \to \infty} g^\nu(\omega,y) = y_{\tau(\omega)}$. The second claim can be proved similarly.
\end{proof}

The next theorem is of independent interest of the Skorokhod topology and it will be crucial for the main results of the paper. In general, $D(\Gamma)$ in the theorem is not closed-valued in the Skorokhod topology, since sequences therein need not converge pointwise.  By $\tau_r$ we denote the topology on $\reals$ generated by the right-open intervals $\{[s, t) \mid s<t\}$.
\begin{theorem}\label{thm:D(S)}
Assume that $\Gamma:\Omega \times [0,T] \tos \reals^d$ is a  $\F \otimes \B([0,T])$-measurable closed convex-valued mapping. Then, $D(\Gamma):\Omega \tos D$ has a measurable graph.
\end{theorem}
\begin{proof}
Let $(t^\nu)_{\nu=1}^\infty$ be a dense sequence in $[0,T]$ containing $T$. For each rational vector $q\in\mathbb Q^{1+d}$, let $H_{q}=\{x\in\reals^d\mid (q_1,\dots,q_d)\cdot x\le q_0\}$ be the associated ``rational'' half-space. By \cite[Theorem 14.3(i)]{rw98}, the sets $A_{q}:=\{(\omega,t)\mid \Gamma_t(\omega)\subset H_{q}\}$, are measurable, so, by time-reversing the argument in the proof of \cite{hwy92}[Theorem 4.2], each
\[
t^{\nu}_q:=\sup\{t | t \le t^\nu, t\in A_{q}\}
\]
is a random time. By the measurable selection theorem, for each $t^{\nu}_q$, there exists random times $t^{\nu,j}_q\in A_q$ such that $t^{\nu}_q-1/j\le t^{\nu,j}_q \le t^{\nu}_q$. We redefine $t^{\nu,j}_q=t^{\nu}_q$ on $\{t^{\nu}_q\in A_q\}$. It suffices to show
\[
D(\Gamma) = \{y\in D \mid y_{t^{\nu,j}_q} \in  \Gamma_{t^{\nu,j}_q} \forall\ q,\nu,j\},
\]
since the right side has a measurable graph, by Lemma~\ref{lem:jointmeas}. In the rest of the proof, we may argue pathwise, so we fix $\omega$.
	
We show first that, for each $q$, $(t^{\nu,j}_q)_{j,\nu=1}^\infty$ is $\tau_r$-dense in $A_q$. Let $t\in A_q$. If there exists $t^\nu>t$ such that $(t,t^\nu)\cap A_q=\emptyset$, then $t=t^{\nu}_q$. Otherwise, there is a subsequence $t^{\nu'}\searrow t$ such that $(t,t^{\nu'})\cap A_q\ne\emptyset$ for each $\nu'$. Thus $t<t^{\nu'}_q\le t^{\nu'}$. For $j(\nu')$ large enough, $t \le t^{\nu',j(\nu')}_q \le t^{\nu'}_q \le t^{\nu'}$, so we have eventually $t^{\nu',j(\nu')}_q\searrow t$.
	
Assume now that $y\in D$ satisfies $y_{t^{\nu,j}_q}\in \Gamma_{t^{\nu,j}_q}$ for each $t^{\nu,j}_q$. Fix $t$ and choose $q$  such that $\Gamma_t\subset H_q$ (note that $H_q=\reals^d$ for $q=0$). Since $t \in A_{q}$ and $(t^{\nu,j}_q)_{\nu,j=1}^\infty$ is $\tau_r$-dense in $A_{q}$, there is a sequence $(t^{\nu_i,j_i}_q)$ converging to $t$  in $\tau_r$ such that $t^{\nu_i,j_i}_q\in A_{q}$. Thus
\[
y_t=\lim y_{t^{\nu_i,j_i}_q}\in H_q.
\]
We obtained
\[
y_t \in \bigcap_q\left\{ H_q \midb \Gamma_t\subset H_q\right\},
\]
where the right side equals $\Gamma_t$, since every closed convex set is an intersection of rational half-spaces containing the set. 
\end{proof}

The following corollary extends \cite[Lemma~5]{pp18d}, where pathwise inner semicontinuity of $S$ was assumed, with a completely different proof. The subspace $C\subset D$ of continuous functions is closed in $D$ and its relative topology w.r.t. the Skorokhod topology is generated by the supremum norm.
\begin{corollary}
Assume that $\Gamma$ is a measurable closed convex-valued  stochastic mapping. Then $C(\Gamma)$ is measurable and closed convex-valued.
\end{corollary}
\begin{proof}
Since $\Gamma$ is closed convex-valued, $C(\Gamma)$ is closed convex-valued as well. By Theorem~\ref{thm:D(S)}, $D(\Gamma)$ has a measurable graph, so $C(\Gamma)$ has a measurable graph.  For complete $\F$, graph measurable and closed nonempty-valued mappings to Polish spaces are measurable \cite[Theorem III.30]{cv77}. If $C(\Gamma)$ is not nonempty-valued, we may apply the above with $\Omega$ replaced by $\dom C(\Gamma)$, which is $\F$-measurable by the projection theorem.
\end{proof}

\section{Integral functionals of \cadlag\ processes}\label{sec:ifcp}
We denote by $L^0_+(M)$ the space of nonnegative random Radon measures which are measurable in the sense of  \cite[VI.86]{dm82} and by $L^1_+(M)$ elements  $\mu \in L^0_+(M)$ such that the random variable $\mu(\cdot,[0,T])$ belongs to the space $L^1(P)$ of $P$ integrable functions. A random measure $\mu \in L^{1}_+(M)$ is  optional if $E \int v d \mu = E \int \op v d \mu$ for each bounded measurable process $v$, where $\op v$ is the optional projection of $v$. For any subset $A \subset M$ we denote by $L^{1}_+(A)$ elements $\mu$ of $L^1_+(M)$ such that $\mu(\omega,) \in A$ almost surely. We set $L^{1}(A)=L^{1}_+(A)-L^{1}_+(A)$. 

A normal integrand $h:\Omega\times[0,T]\times\reals^d\rightarrow\ereals$ on $\reals^d$ is {\em optional} (resp. {\em predictable}) if its epigraphical mapping is optional (resp. predictable). In what follows, we use the qualifier $a.s.e.$ (almost surely everywhere) for a property satisfied outside an evanescent set.

Let $\D$ be a linear subspace of adapted $\cadlag$ processes.
\begin{assumption}\label{ass:ifocp}
The convex normal integrand $h$ is optional and
\begin{enumerate}	
\item $(t,x)\to h_t(x,\omega)$ satisfies the Assumption \ref{ass:clS}  almost surely,
\item there exists an optional process  $x$ and  nonnegative measurable process $\alpha$ with $E\int |x|  | y |  d\mu<\infty$ for all $y\in\D$ and  $E\int\alpha d\mu<\infty$ such that
\[
h_t(y,\omega)\ge x_t(\omega)\cdot y-\alpha_t(\omega)\quad \text{a.s.e.}
\]
\end{enumerate}
\end{assumption}

Throughout, $ L^{0}(D)$ denotes the space of \cadlag-valued random variables $y:\Omega \to D$ which are $(\F,\B(D))$-measurable, $ L^{0}(D,S)$ the elements of $L^0(D)$ that are almost surely selections of $S$, and $L^{\infty}(D)$ the elements $y \in L^{0}(D)$ such that $\left\| y \right\|_{\infty}\in L^\infty$. We set $\D(S)=\D\cap L^0(D,S)$ and $L^{i}(D,S)=L^{i}(D) \cap L^0(D,S)$ for $i \in \{0,\infty\}$. Let $\D^\infty$ be the space of bounded adapted \cadlag\ processes and $\D^{\infty}(S):= \D^\infty\cap L^0(D,S)$.



\begin{theorem}\label{thm:ifcp}
Let $\D$ be a space of adapted \cadlag\ processes containing $\D^\infty$. Under Assumption~\ref{ass:ifocp},
\[
\inf_{ y  \in \D} \left[ E I_{ h }(y) + \delta_{\D(S)}(y) \right]=E\left[ \int \inf_{x\in\reals^d} h (x)d\mu\right]
\]
as soon as the left side is less than $+\infty$.
\end{theorem}
\begin{proof}
Assume first that $E I_{ h }$ is finite for some $y\in\D^\infty(S)$. By Lemma~\ref{lem:ni1} $I_{ h }$ is $\F \otimes \B(D)$-measurable while $\delta_{D(S)}$ is $\F \otimes \B(D)$-measurable due to Theorem \ref{thm:D(S)}. Then, $I_{ h }+\delta_{D(S)}$ is $\F \otimes \B(D)$-measurable and 
\begin{equation*}
\begin{split}
\inf_{ y  \in L^\infty(D)} E[I_{ h }( y ) + \delta_{D(S)}(y)]
&= E\left[ \inf_{x\in D} \left\lbrace \int h (x)d\mu + \delta_{D(S)}(x)\right\rbrace \right] \\
&= E\left[ \inf_{x\in D(S)}  \int h (x)d\mu \right] \\
&= E\left[ \int \inf_{x\in\reals^d}  h (x)d\mu\right],
\end{split} 
\end{equation*}
where the first equality follows from Theorem~\ref{thm:ifs}, the second is clear,  and the third follows from Theorem~\ref{thm:ifcadlag}. On the other hand 
\begin{equation*}
\begin{split}
\inf_{ y \in\D^\infty(S)} EI_{ h }( y ) 
&=  \inf_{ y  \in L^\infty(D)}\left[  EI_{ h }(\op  y ) + \delta_{D(S)}(\op y) \right]\\
&\leq \inf_{  y  \in L^\infty(D,S)} EI_{ h }( y ) \\
&\leq  \inf_{ y \in\D^\infty(S)} EI_{ h }( y ),
\end{split}
\end{equation*}
where in the first equality we used the fact that optional projections of bounded $\cadlag$ processes are bounded and $\cadlag$ \cite[Theorem VI.47]{dm82} and that optional projections of selections are again selections; see \cite[Corollary~4]{kp17}. In the first inequality we  applied  Jensen's inequality Lemma \ref{lem:jin} and Jensen's inequality for optional set-valued mappings \cite[Theorem~9]{kp17}.

Assume now that $E I_{ h }$ is finite for arbitrary $\bar y\in\D(S)$. Then
\[
\bar h_t(y,\omega):=h_t(y+\bar y_t(\omega),\omega)
\]
satisfies Assumption~\ref{ass:ifocp} and the corresponding integral functional is finite at the origin. Thus the result follows from the first part.
\end{proof}

Next, we formulate a variant of the above result for adapted $\caglad$ processes.  A random measure $\tilde \mu \in L^{1}_+(M)$ is  predictable if $E \int v d \mu = E \int \pp v d \mu$ for each bounded measurable process $v$, where $\pp v$ is the predictable projection of $v$. Let $\tilde \mu  \in L^{1}_+(M)$ be a nonnegative predictable random Radon measure, $\tilde h:\Omega\times[0,T]\times\reals^d\to\ereals$ a convex normal integrand on $\reals^d$.  Let $\D_l$ be a linear subspace of adapted $\caglad$ processes.
\begin{assumption}\label{ass:ifpcp}
The convex normal integrand $h$ is predictable and
\begin{enumerate}	
\item $(t,x)\to h_t(x,\omega)$ satisfies the Assumption \ref{ass:clSp} almost surely,
\item there exists a predictable process $x$ and nonnegative measurable process $\alpha$ with $E\int |x|| y |  d\tilde \mu<\infty$ for all $y\in\D_l$ and  $E\int\alpha d\mu<\infty$ such that
\[
\tilde h_t(y,\omega)\ge x_t(\omega)\cdot y-\alpha_t(\omega)\quad \text{a.s.e.}
\]
\end{enumerate}
\end{assumption} 

We denote the space of  bounded adapted \caglad\ processes by $\D_l^\infty$ and by $\D_{l}(\tilde S)$ the elements of $\D_{l}$ which are almost surely selections of $\tilde S$.

\begin{theorem}\label{thm:ifpcp}
Let $\D_l$ be a space of adapted $\caglad$ processes containing $\D_l^\infty$. Under Assumption~\ref{ass:ifpcp},
\[
\inf_{ y  \in\D_l} [ E I_{ \tilde h }(y) + \delta_{\D_l(\tilde S)}(y)]=E\left[ \int \inf_{x\in\reals^d} \tilde h (x)d \tilde \mu\right]
\]
as soon as the left side is less than $+\infty$.
\end{theorem}
\begin{proof}
Recalling that predictable projections of bounded $\caglad$ processes are bounded and $\caglad$ \cite[Theorem VI.47]{dm82}, the theorem is proved similarly to Theorem~\ref{thm:ifcp}.
\end{proof}

 The next assumption gives  a necessary compatibility condition so that we get an interchange rule for 
\begin{equation*}
\hat F(y)= E\left[I_h(y) + I_{\tilde h}(y_-)\right]+\delta_{\D(S)}(y)+\delta_{\D_l(\tilde S)}(y_-).
\end{equation*}
\begin{assumption}\label{ass:ifcp}
The predictable measure $\tilde\mu$ is purely atomic and, for every $\epsilon>0$, $ y  \in\dom EI_{ h } \cap  \D(S)$, and $\tilde y\in \dom EI_{\tilde h}\cap \D_l(\tilde S)$, there exists $\check  y  \in \D(S)$  and  $\hat  y  \in \D_l(\tilde S)$ with
\begin{enumerate}
\item $EI_{ h }(\check  y )\le EI_{ h }( y )+\epsilon$ and $\check y_-\in\dom EI_{\tilde h}\cap \D_l(\tilde S)$,
\item $EI_{\tilde h }(\hat  y )\le EI_{ \tilde h }( \tilde y )+\epsilon$ and $\hat y_+\in\dom EI_{h}\cap \D(S)$.
\end{enumerate}
\end{assumption}

A class of adapted $\cadlag$ processes $\D$ is called {\em solid} if $y\in\D$ whenever $|y|\le |\bar y|$ for some $\bar y\in\D$, and {\em max-stable} if $\max \{|y^1|,|y^2|\}\in\D$ whenever $y^1,y^2\in\D$.

\begin{theorem}\label{thm:ircp}
Let $\D$ be a solid max-stable space containing $\D^\infty$ and let $\D_l=\{y_-\mid y\in \D\}$.  Under Assumptions~\ref{ass:ifocp}, \ref{ass:ifpcp} and \ref{ass:ifcp},
\[
\inf_{ y \in\D} \hat F(y)=E\left[\int \inf_{x\in\reals^d} h(x)d\mu + \int \inf_{x\in\reals^d}\tilde h(x)d \tilde\mu\right]
\]
as soon as the left side is less than $+\infty$.
\end{theorem}

\begin{proof}
Let $\epsilon>0$. By Theorems~\ref{thm:ifcp} and \ref{thm:ifpcp} and Assumption~\ref{ass:ifcp}, there exist $y, \tilde y  \in\D(S)$  such that $\tilde  y_{-}  \in \D_l(\tilde S)$, $\tilde y  \in  \dom EI_{h}\cap\D(S)$, $y_-\in\dom EI_{\tilde h}\cap \D_l(\tilde S)$  and
\begin{equation}\label{eq:ifcp2p}
EI_{ h }( y )+ E I_{\tilde h }(\tilde  y _{-})\le E \left[ \int \inf_{x\in\reals^d} h _t(x)d\mu  + \int \inf_{x\in\reals^d}\tilde h _t(x)d\tilde\mu \right] + \epsilon.
	\end{equation}
	
In the next construction we take $T=\infty$, the case $T < \infty$ being more simple. Since $\tilde\mu$ is purely atomic and predictable, the process $w$ defined by
\[
w_t:=\int_{[0,t]} d{ \tilde\mu}
\]
is predictable and purely discontinuous whose jump times belong to the set $A :=\{\Delta w \ne 0\}$. 	For each $\nu=1,2,\dots$ let $A_\nu:=\{\Delta w \ge\nu^{-1}\}$. For each $n=0,1,\ldots,$ the predictable set $A_{\nu} \cap \llparenthesis n,n+1 \rrbracket$ is a union of graphs of the elements of an increasing sequence $(\tau^{\nu,n,j})_{j=1}^\infty$ of predictable times with graphs in $\llparenthesis n,n+1 \rrbracket$. Fix a decreasing sequence of positive real numbers converging  to zero  $\{\epsilon_n\}_{n \in  \nn}$.  Let $\{\sigma^{\nu,n,j,k}\}_{k \in \nn}$ be an announcing sequence for $\tau^{\nu,n,j}$ as in Lemma \ref{lablemmaAlternativeannouncing} below in the interval $(n,n+1]$ with sequence $\{\frac{1}{2^j} \epsilon_k\}_{k \in \nn}$.  We define a process $y^{\nu,k}$ by 
\begin{equation}\label{eq:pasting}
y^{\nu,k}_t(\omega)=\begin{cases}
\tilde  y _t(\omega)\ \ &\text{if } (\omega,t) \in \bigcup_{n=0}^{\infty} \left\{( n,n+1 ] \cap \bigcup_{j=1}^{\infty}[ \sigma^{\nu,n,j,k},\tau^{\nu,n,j} ) \right\} ,\\
y _t(\omega)\ \ &\text{otherwise}. 
\end{cases}
\end{equation}
Outside a null set,  the process $y^{\nu,k}$ has the following properties  
\begin{align}
y^{\nu,k}  & \mbox{ has $\cadlag$ paths}, \label{lim:1}\\
\lim_{k \to \infty} y^{\nu,k} &= y, \mbox{ a.s.e. stationarily}, \label{lim:2}\\ 
\lim_{k \to \infty} y^{\nu,k}_{-} &= y_{-}, \mbox{ a.s.e. stationarily on } A_{\nu}^c, \label{lim:3}\\
\lim_{k \to \infty} y^{\nu,k}_{-} &= \tilde y_{-}, \mbox{ a.s.e. stationarily on } A_{\nu} \label{lim:4},
\end{align}
which we verify next.
  
Fix $\nu$, $n$ and $k$. Let $D_j:=\{\sigma^{\nu,n,j,k}< \infty \}$ and  $B_j:=\{\tau^{\nu,n,j} =\infty \}$. We have $\rpr(D_j \cap B_j) \leq \frac{\epsilon_k}{2^j}$ and so $C:=\limsup_{j \to \infty} D_j \cap B_j$ is a null set due to Borel-Cantelli Lemma. Let $N$ be a null set such that $w_{T}(\omega)<\infty$ for $\omega \in \Omega/N$. Let $M$ be a null set such that on $\Omega /M $ the stopping times $\sigma^{\nu,n,j,k}$ converge to $\tau^{\nu,n,j}$. Let
\[
p^{\nu,n,k}(\omega):= \left\lbrace  t \in [n,n+1] \mid (\omega,t) \in  \bigcup_{j=1}^{\infty}  [\sigma^{\nu,n,j,k}, \tau^{\nu,n,j} ) \right\rbrace.
\]
We get \eqref{lim:1} once we  show that $p^{\nu,n,k}(\cdot)$ is a finite union of  semiopen sets $[a,b)$ on  $\Omega / (C \cup N \cup M)$, and it is actually included on $(n,n+1]$.  For  $\omega \in \Omega / (C \cup N \cup M)$ there exists $j_0$ such that for all $j \geq j_0$ we have $\omega \notin D_j \cap B_j$, since $\omega$ is not an element of $C$. There are two alternatives:
\[
\sigma^{\nu,n,j,k}(\omega) = \infty \mbox{ or } \tau^{\nu,n,j}(\omega) <\infty. 
\]
The second alternative can only happen for a finite number of indexes $j \geq j_0$, since $\omega$ is not an element of $N$.  In the first alternative we clearly have that the interval  $[\sigma^{\nu,n,j,k}(\omega), \tau^{\nu,n,j}(\omega))$ is empty. Thus  
\[
p^{\nu,n,k}(\omega)=\bigcup_{j=1}^{\tilde j}  [\sigma^{\nu,n,j,k}(\omega), \tau^{\nu,n,j}(\omega))
\]
for some $\tilde j$. The set $p^{\nu,n,k}(\omega)$ is included on $(n,n+1]$ since each $\sigma^{\nu,n,j,k}$ is strictly greater than $n$. Now \eqref{lim:1} is established. Other properties of $p^{\nu,n,k}(\cdot)$ on   $\Omega / (C \cup N \cup M)$ are  $p^{\nu,n,k+1} \subset p^{\nu,n,k}$ and $\cap_{k \in \nn} p^{\nu,n,k} = \emptyset$.  Thus \eqref{lim:2} holds.
	
Let 
\[
q^{\nu,n,k}(\omega) := \left\lbrace  t \in (n,n+1] \mid (\omega,t) \in ( n,n+1 ] \cap \bigcup_{j=1}^{\infty}  ( \sigma^{\nu,n,j,k}, \tau^{\nu,n,j} ] \right\rbrace.
\]
Since $p^{\nu,n,k}$ is almost surely a finite union, we have, for $t \in (n,n+1]$
\begin{equation}\label{eq:pastingl}
y^{\nu,k}_{t-}(\omega)=
\left\lbrace
\begin{array}{cl}
\tilde y_{t-}(\omega)  & \mbox{if } t \in q^{\nu,n,k}(\omega)\\
y_{t-}(\omega)  & \mbox{otherwise}.
\end{array}
\right.
\end{equation}
Since also  $q^{\nu,n,k+1} \subset q^{\nu,n,k}$ and $\bigcap_{k=1}^{\infty} q^{\nu,n,k}$ is equal to $\bigcup_{j=1}^{\infty} \llbracket \tau^{\nu,n,j} \rrbracket$ we get  \eqref{lim:3} and \eqref{lim:4}. Now \eqref{lim:2} gives 
\[
\lim_{k\rightarrow\infty} h ( y ^{\nu,k})=  h ( y ), \quad P\text{-a.s. for all } t \in [0, T ],
\]
and, for all $(\omega,t)\in A$,  \eqref{lim:3} and \eqref{lim:4} give 
\[
\lim_{k\rightarrow\infty} \tilde h ( y ^{\nu,k}_-) = \tilde h ( y _{-} 1_{A_\nu^C}+\tilde  y _{-} 1_{A_\nu}).
\]
By the definition of $A_{\nu}$ it follows that 
$\lim_{\nu\rightarrow\infty}\tilde h ( y _{-} 1_{A_\nu^C}+\tilde  y _{-} 1_{A_\nu}) = \tilde h (\tilde  y _{-})$. We have, by \eqref{eq:pasting} and \eqref{eq:pastingl},
\[
h ( y ^{\nu,k}) + \tilde h ( y ^{\nu,k}_-) \le\max\{h (\tilde  y ), h ( y )\}+\max\{\tilde h (\tilde  y _{-}),\tilde h ( y _{-})\},
\]
where the right side is integrable, by the existence of the lower bounds  in Assumptions~\ref{ass:ifocp} and \ref{ass:ifpcp} and by the fact that $E[I_{ h }( y )+I_{ h }(\tilde  y )+  I_{\tilde h }( y _{-})+  I_{\tilde h }(\tilde  y _{-})]<\infty$. By Fatou's lemma,
\[
\limsup_{\nu\rightarrow\infty}\limsup_{k\rightarrow\infty} \left\lbrace EI_{ h }( y ^{\nu,k})+E I_{\tilde h }( y ^{\nu,k}_-)\right\rbrace \le EI_{ h }( y )+E  I_{\tilde h}(\tilde  y _-)
\]
which in combination with \eqref{eq:ifcp2p} finishes the proof.
\end{proof}


\section{Conjugates of integral functionals}\label{sec:main}

This section presents the main results of the article. That is, we characterize the convex conjugates and the subdifferentials of  $F$ and  $\hat F$  defined respectively in \eqref{eq:F} and \eqref{eq:hatF}.  The results are based on the interchange rules developed in the previous sections. We start by specifying appropriate spaces $\D$ and $\hat \M$, of adapted \cadlag\ processes and random measures in duality.

From now on, we assume that $\D$ and $\D_l$ satisfy the conditions in Theorem \ref{thm:ircp} (i.e. $\D$ is a solid max-stable space containing $\D^\infty$ and $\D_l=\{y_-\mid y\in \D\}$), $\hat\M$ is a subspace of 
\[
\{(u,\tilde u)\in L^1(M) \times L^1(\tilde M)\mid u  \mbox{ optional, } \tilde u \mbox{ predictable}\}
\]
 containing $\hat\M^\infty:=\{(u,\tilde u)\in L^\infty(M)\times L^\infty(\tilde M)\mid u  \mbox{ optional, } \tilde u \mbox{ predictable}\}$ and that, for all $y\in\D$ and $(u,\tilde u)\in\hat\M$,
\begin{equation*}
E\left[ \int |y|d|u|+\int|y_-|d|\tilde u|\right] <\infty.
\end{equation*}
In particular, the bilinear form
\[
\langle y,(u,\tilde u)\rangle :=E\left[\int y du + \int y_- d\tilde u\right]
\]
is well-defined on $\D\times\hat\M$. We equip $\D$ and $\hat\M$ with topologies compatible with this bilinear form. The next example shows how many familiar Banach spaces of adapted $\cadlag$ processes together with their duals fit in our setting.
\begin{example}
For $p\in(1,\infty)$, let $\D^p$ be the space of adapted $\cadlag$ processes whose pathwise supremum belongs to $L^p$. When endowed with the norm $\|y\|_{\D^p}= (E\|y\|^p)^{1/p}$, $\D^p$ is Banach space whose dual may be identified with
\[
\hat \M ^q=\{(u,\tilde u)\in L^q(M) \times L^q(\tilde M)\mid u  \mbox{ optional, } \tilde u \mbox{ predictable}\}
\]
for $1/p+1/q=1$; see, e.g., \cite{dm82}. These dual pairs evidently satisfy our assumptions as $\D$ and $\hat\M$. In this setting, $L^p$ can be replaced by the Morse heart of an appropriate Orlicz space; see \cite{ara14}. 

Let $\D^1$ be the space of adapted $\cadlag$ processes of class $(D)$. When endowed with the norm $\sup_{\tau\in\T} E|y_\tau|$, $\D^1$ is a Banach space whose dual can be identified with $\hat \M^\infty$; see, e.g., \cite{pp18b}. This dual pair satisfies our assumptions as well. This setting extends to appropriate spaces with the so called ``Choquet property'' \cite{pp18b}.
\end{example}


We denote $\M=\{ u \mid (u,0)\in\hat\M\}$. Recall the definition of $J_{h^*}$ in \eqref{eq:Jh}.

\begin{lemma}\label{lem:lsc}
We have property 2 in Assumption~\ref{ass:ifocp}  if and only if $EJ_{h^*}(u)$ is finite for some $u\in\M$.  In this case, for every $y\in\D$
\begin{equation}
\label{eq:DualJhstar}
\sup_{u\in\M}\left\{E \int ydu - EJ_{h^*}(u)\right\}=F(y),
\end{equation}
and in particular, $F$ is lower semicontinuous on $\D$.  We have an analogous result for $EJ_{\tilde h^*}(\cdot)$ with respect to Property 2  in Assumption~\ref{ass:ifpcp} and the class $\tilde \M =\{ \tilde u \mid (0,\tilde u)\in\hat\M \}$.
\end{lemma}
\begin{proof}
We only prove the claims for $EJ_{h^*}(u)$, since for $EJ_{\tilde h^*}(u)$ the proof is similar. Assume that Property 2 in Assumption~\ref{ass:ifocp} holds. The measure $u$  defined by  $du= x d\mu$ is an element of $\M$ and  $E J_{h^*}(u)= E  \int h^*(x) d \mu \leq E \int \alpha d \mu < \infty$. Conversely, let $u \in \M$ be such that $EJ_{h^*}(u)$ is finite. Then Property 2 in Assumption~\ref{ass:ifocp} holds  with $\alpha = (h^*(\frac{d u^a}{d\mu}))^+$ and $x= \frac{d u^a}{d\mu}$.

Take $y\in\D$. Assume first that $EJ_{h^*}$ is finite at the origin. Let $\theta$ be the optional measure on the optional $\sigma$-algebra $\O$ given by $\theta(A)=E\int 1_A d\mu$. Then, 
\begin{equation}\label{eq:DualJhstarA1}
\begin{split}
\sup_{u\in\M}\left\{E\int ydu - EJ_{h^*}(u)\right\} &\ge \sup_{w\in L^\infty(\Xi,\O,\theta)}\left\{\int y wd\theta - \int h^*(w)d\theta\right\}\\
&=\int h(y)d\theta\\
&= EI_h(y),
\end{split}
\end{equation}
where the first equality follows from \cite[Theorem 14.60]{rw98}, since the process  identically equal to zero belongs to $L^\infty(\Xi,\O,\theta)$. For a stopping time $\tau$, $h_{\tau}$ is a normal integrand; see the discussion after Theorem~2 in \cite{kp17}. Assume that $\theta( \llbracket \tau \rrbracket)=0$. By \cite[Theorem 14.60]{rw98} again, and by \eqref{eq:domrec}, we get
\begin{equation}\label{eq:DualJhstarA2}
\sup_{u\in\M}\{E \int ydu - EJ_{h^*}(u)\}\ge \sup_{\eta\in L^\infty(\F_\tau)} E[\eta\cdot y_\tau -(h^*_\tau)^\infty(\eta) ] = E\delta_{S_\tau}(y_\tau). 
\end{equation}
As a consequence, for $y$ such that the left hand side of \eqref{eq:DualJhstar} is finite, we get  $y\in\D(S)$ by \eqref{eq:DualJhstarA1} and \eqref{eq:DualJhstarA2}.  Conversely, if $y$ is not an element of $\D(S)$ then \eqref{eq:DualJhstarA2} yields that the left hand side of \eqref{eq:DualJhstar} is infinite. Indeed, the set $\{(\omega,t) \mid S_t(\omega) \cap \{y\} =\emptyset \}$ is optional and we conclude with the optional section theorem.

 We have shown that
\[
\sup_{u\in\M}\left\{ E \int ydu - EJ_{h^*}(u)\right\} \ge F(y)
\]
while the opposite inequality follows from Fenchel's inequality. If $EJ_{h^*}$ is not finite at the origin but at $\bar u$, we apply the above to  $u \to EJ_{h^*}(u + \bar u)$.
\end{proof}

\begin{assumption}\label{ass:ifcp2}
$\hat F$ is finite at some point. The predictable measure $\tilde\mu$ is purely atomic and, for every $\epsilon>0$, $ y  \in\dom EI_{ h } \cap  \D(S)$ and $\tilde y\in \dom EI_{\tilde h}\cap \D_l(\tilde S)$ and $(u,\tilde u)\in\hat\M$, there exists $\check  y  \in \D(S)$ and $\hat  y  \in \D_l(\tilde S)$ with
\begin{enumerate}
\item  $EI_{ h }(\check  y )+E\int \check ydu \le EI_{ h }( y ) + E\int ydu +\epsilon$ and $\check y_-\in\dom EI_{\tilde h}\cap \D_l(\tilde S)$,
\item  $EI_{\tilde h }(\hat  y ) +E\int \hat yd\tilde u \le EI_{ \tilde h }( \tilde y ) +E\int \tilde yd\tilde u+\epsilon$ and $\hat y_+\in\dom EI_{h}\cap \D(S)$.
\end{enumerate}
\end{assumption}

The following is our first main result.
\begin{theorem}\label{thm:cifcp}
Under Assumptions~\ref{ass:ifocp}, \ref{ass:ifpcp} and \ref{ass:ifcp2}, $\hat F$ is a proper lower semicontinuous convex function on $\D$,
\[
\hat F^*(u,\tilde u)=E\left[J_{h^*}(u)+J_{\tilde h^*}(\tilde u)\right]
\]
and $(u,\tilde u)\in\partial \hat F(y)$ if and only if almost surely,
\begin{align*}
du/d\mu &\in \partial h(y)\quad \mu\text{-a.e.},\\
du/d|u^s| &\in \partial^s h(y)\quad |u^s|\text{-a.e.},\\
d\tilde u/d\tilde\mu &\in \partial \tilde h(y_-)\quad \mu\text{-a.e.},\\
d\tilde u/d|\tilde u^s| &\in \partial^s \tilde h(y_-)\quad |\tilde u^s|\text{-a.e.}
\end{align*}
\end{theorem}

\begin{proof}
We note first that $F$ is proper and lsc. Indeed, applying Lemma~\ref{lem:lsc} with $\M$ and $\tilde \M$, we see that $\hat F$ is lsc and it never takes the value $-\infty$, while finitess at some point is assumed explicitly in Assumption~\ref{ass:ifcp2}. We have
\begin{align*}
&\hat F^*(u,\tilde u)\\
&=\sup_{y\in\D}\{\langle y,(u,\tilde u) \rangle -\hat F(y)\}\\
&=-\inf_{y\in\D}  \left\lbrace  E \left[ I_h(y)+I_{\tilde h}(y_-)- \int ydu-\int y_- d\tilde u  \right] +\delta_{\D(S)}(y) +\delta_{\D_l(\tilde S)}(y_-)\right\rbrace .
\end{align*}

Let $\tilde u= \tilde u^a + \tilde u^s$ be the Lebesgue decomposition of $\tilde u$ with respect to $\tilde \mu$ and $\tilde A$ be a predictable set such that
\[
E\int 1_{\tilde A^C \cap B}  d\tilde\mu=E\int 1_{\tilde A \cap B} d|\tilde u^s|=0
\]
for any predictable set $B$; see  \cite[Theorem 5.15]{hwy92} and \cite[Theorem 2.1]{Delbaen1995}.
Defining $\check\mu :=|\tilde u^s|+ \tilde \mu$  
\[
\check h(y) := \begin{cases}
\tilde h(y) - y\cdot \frac{d \tilde u^a}{d \tilde  \mu}\quad&\text{ on } \tilde A,\\
\delta_{\tilde S}(y)- y\cdot \frac{d \tilde u^s}{d  |\tilde u^s|} \quad&\text{ on } \tilde A^c,
\end{cases}
\]
and $\check S_t:=\cl\dom \check h_t= \tilde S_t$, we have 
\[
E\left[I_{\tilde h}(y_-)-\int y_-d \tilde u \right]+\delta_{\D_l(\tilde S)}(y_-)=E\int \check h(y_-)d\check{\mu}+\delta_{\D_l(\check S)}(y_-).
\]
Likewise, let $u= u^a +  u^s$ be the Lebesgue decomposition of $u$ with respect to $\mu$ and $A$ an an optional  set such that
\[
E\int 1_{A^C \cap B}  d \mu=E\int 1_{A \cap B} d|u^s|=0
\]
for any optional set $B$.  Defining $\hat\mu :=|u^s|+\mu$
\[
\hat h(y) := \begin{cases}
 h(y) - y\cdot \frac{d  u^a}{d   \mu}\quad&\text{ on }  A,\\
\delta_{ S}(y)- y\cdot \frac{d u^s}{d  | u^s|} \quad&\text{ on }  A^c,
\end{cases}
\]
and $\hat S_t:=\cl\dom \hat h_t=S_t$, we have that
\[
E\left[I_{h}(y)-\int ydu \right]+\delta_{\D(S)}(y)=E\int \hat h(y)d\hat\mu+\delta_{\D(\hat S)}(y).
\]
Recalling \eqref{eq:domrec}, we have
\[
\inf_{y\in\reals^d}\check h_t(y,\omega)=\begin{cases}
 -\tilde h^*_t((d\tilde u/d\tilde\mu)_t(\omega),\omega)\quad&\text{if }(\omega,t)\in \tilde A,\\
 - (\tilde h_t^*)^\infty((d\tilde u/d|\tilde u^s|)_t(\omega),\omega)\quad&\text{otherwise}
\end{cases}
\]
and similarly for $\hat h$. It is straightforward (although slightly tedious) to verify the assumptions in Theorem~\ref{thm:ircp} for $\hat h$ and $\check h$, which then gives the conjugate formula.

To prove the subgradient formula, let $y\in\dom \hat F$ and $(u,\tilde u)\in\hat\M$. By Fenchel's inequality, almost surely,
\begin{align*}
h(y)+h^*(du/d\mu)&\ge y\cdot(du/d\mu)\quad\mu\text{-a.e.,}\\
(h^*)^\infty(du/d|u^s|) &\ge y\cdot (du/d|u^s|)\quad|u^s|\text{-a.e.,}\\
\tilde h(y_-)+\tilde h^*(d\tilde u/d\tilde\mu)&\ge y_-\cdot(d\tilde u/d\tilde \mu)\quad\tilde\mu\text{-a.e.,}\\
(\tilde h^*)^\infty(d\tilde u/d|\tilde u^s|) &\ge y_-\cdot (d\tilde u/d|\tilde u^s|)\quad|\tilde u^s|\text{-a.e.}
\end{align*}
We have $(u,\tilde u)\in\partial \hat F(y)$ if and only if $\hat F(y)+ \hat F^*(u,\tilde u)=\langle y,(u,\tilde u)\rangle$ which  by the conjugate formula, is equivalent to having the above inequalities satisfied as equalities which in turn is equivalent to the stated pointwise subdifferential conditions.
\end{proof}

The following is our second main result. Sufficient conditions for predictability of $\tilde S$ in the theorem will be given in Theorem~\ref{thm:vectilde} below. In addition to the conditions for $\D$ stated at the beginning of this section, we also assume  in the next result that $\D$ is a subspace  of class $(D)$ processes.  We denote the open unit ball in $\reals^d$ by $\uball$.
\begin{theorem}\label{thm:cifcpcor}
Assume that $EI_h$ is finite on $\D(S)$, Assumption~\ref{ass:ifocp} holds and that
\begin{equation*}
\tilde S_t(\omega) :=\cl\{y_{t-} \mid y \in D(S(\omega))\}
\end{equation*}
defines a predictable set-valued mapping. Then 
\[
F^*(u,\tilde u)=EJ_{h^*}(u)+EJ_{\sigma_{\tilde S}}(\tilde u).
\]
Moreover, $F^*$ is the lower semicontinuous hull of $(u,\tilde u)\to EJ_{h^*}(u)+\delta_{\{0\}}(\tilde u)$.
\end{theorem}
\begin{proof}
Let $\tilde h=\delta_{\tilde S}$ so that $F=\hat F$. Assumption~\ref{ass:ifpcp} holds, so by Theorem~\ref{thm:cifcp}, it suffices to show that  Assumption~\ref{ass:ifcp2} is satisfied. Indeed, the last claim then holds as well, by Lemma~\ref{lem:lsc} and the biconjugate theorem. 
	
If $y\in \D(S)$, then $y_-\in\D_l(\tilde S)$, and property 1 in Assumption~\ref{ass:ifcp2} holds. To show property 2, take  $\tilde y \in\D_l(\tilde S)$, $\epsilon>0$ and $(u,\tilde u)\in\hat\M$. Choose $y\in \D(S)$.
	
Since $\tilde u$ is purely atomic and predictable, the process $w$ defined by
\[ 
w_t:=\int_{[0,t]} (|\tilde y|+|y_-|)d  |\tilde u|
\]
is predictable increasing and purely discontinuous whose jump times belong to the set $A :=\{\Delta w_t\ne 0\}$. Defining $A_\nu:=\{\Delta w_t \ge 1/\nu\}$ and fixing $\nu$ large enough, we get
\begin{equation*}
E\left[\int_{A_\nu^C} (|\tilde y|+|y_-|)d |\tilde u| \right]<\epsilon/2.
\end{equation*}
Here $A_\nu$ is supported on a union of graphs of an increasing disjoint sequence $(\tau^{j})$ of predictable times.  Take $\alpha>0$ such that $E[| \tilde u |(A)] <\alpha$. Let 
\[
\begin{split}
\hat S_t(\omega)&=\{x\in\reals^d\mid |x|\le |y_t(\omega)|+|\tilde y_{t+}(\omega)|+1\}\\
\Gamma^j(\omega)&=\{ z\in D\mid z_{\tau^j(\omega)-} \in \tilde y_{\tau^j(\omega)}(\omega) + \frac{\epsilon}{\alpha 2^{j+1}} \uball \}.
\end{split}
\]
We have that $D(S)$ and $D(\hat S)$ are graph-measurable by Theorem \ref{thm:D(S)}.  Let $T_j:\Omega  \times \B(D) \to \reals$ be defined by $T_j(\omega,z)= z_{\tau^{j}(\omega)-} - \tilde y_{\tau^{j}(\omega)}(\omega)$. Then $T_j$ is $\F \otimes \B(D)$-measurable by Lemma \ref{lem:jointmeas}. Hence, $\Gamma^j$ is graph-measurable since  $\graph {\Gamma^j}=T^{-1}_j( \frac{\epsilon}{\alpha 2^{j+1}} \uball)$. As a consequence, we see the graph-measurability of  the mapping
\[
\Gamma(\omega):=D(S) \cap D(\hat S) \cap \bigcap_{j} \Gamma^j.
\]

Now we check $\Gamma$ is nonempty-valued. We fix $\omega$. For all $j$ there exists $z^j \in D(S(\omega))$ with $z^j_{\tau^j(\omega)-}\in \tilde y_{\tau^j(\omega)}(\omega) + \frac{\epsilon}{\alpha 2^{j+1}} \uball$ by the definition of $\tilde S$.  For  $\delta>0$ let
\[
z:=\sum_{j} z^j 1_{[\tau^j(\omega)-\delta,\tau^j(\omega))} + y(\omega) (1- 1_{\bigcup_{j} [\tau^j(\omega)-\delta,\tau^j(\omega))}).
\]
The series defining $z$ is a finite sum since $A_{\nu}$ is $\omega$-wise finite. Thus, $z$ is a $\cadlag$ function and it is clear that it is a selection of $S(\omega)$.  We choose $\delta$ in  such a way that $z^j_{t} \in z^{j}_{\tau^j(\omega)-} + \frac{\epsilon}{2^{j+1}} \uball$ and  $\tilde y_{t}(\omega) \in \tilde y_{\tau^j(\omega)}(\omega) + \frac{\epsilon}{2^{j+1}} \uball$ for $t \in [\tau^j(\omega)-\delta,\tau^j(\omega))$. Then $z \in D(\hat S(\omega))$. It is also clear that $z \in \Gamma^j(\omega)$. Thus,  $\Gamma(\omega)$ is nonempty.\\


By \cite[Theorem III.22]{cv77}, there exists a $\cadlag$ selection $z$ of $\Gamma$ which seen as a process is measurable although possibly non-adapted. The bound
\[
|z|\le |y(\omega)|+|\tilde y_{+}(\omega)|+1
\]
implies that $\op z$ exists and belongs to $\D$. The process $z$ satisfies $(\op z)_-=\pp(z_-)$, by \cite[Lemma~4]{pp18b}. Moreover, by \cite[Corollary~4]{kp17}, $\op z\in\D(S)$. Let $(\sigma^{j,\nu})$ be an  announcing sequence for $\tau^j$, where we may assume that $\sigma^{j+1,\nu}\ge \tau^{j}$ for every $j$ and $\nu$. Defining
\[
\hat y^\nu = \sum_j \op z\one_{[\sigma^{j,\nu},\tau^j)} + y\one_{(\bigcup_j[\sigma^{j,\nu},\tau^j))^C},
\]
we have that $\hat y^\nu\in \D(S)$ and hence $\hat y^\nu\in\dom EI_h$. For $\nu$ large enough,
\[
E\int \hat y^\nu_{-} d\tilde u \le E\int \tilde y d\tilde u +\epsilon
\]
which shows property 2 in Assumption~\ref{ass:ifcp2}.
\end{proof}

The following result is an immediate corollary of Theorem~\ref{thm:cifcpcor}.
\begin{corollary}\label{cor:S}
Assume that $S$ is an optional set-valued mapping with
\[
S_t(\omega)=\cl\{y_t\mid y\in D(S)(\omega)\}
\]
and that $\tilde S$ defined by 
\[
\tilde S_t(\omega):=\cl\{y_{t-}\mid y\in D(S)(\omega)\}
\]
is predictable. Then $\D(S)$ is closed and
\[
\sigma_{\D(S)}(u,\tilde u)=EJ_{\sigma_S}(u)+EJ_{\sigma_{\tilde S}}(\tilde u)
\]
as soon as $\D(S)\ne\emptyset$. Moreover, $\sigma_{\D(S)}$ is the lower semicontinuous hull of 
\[
(u,\tilde u)\to EJ_{\sigma_S}(u)+\delta_{\{0\}}(\tilde u).
\]
\end{corollary}

\section{C\'adl\'ag selections of set-valued mappings}\label{sec:SVA}
One of the conditions in our main results above is  
\begin{equation*}
S_t=\cl \{y_t \mid y \in D(S)\}
\end{equation*}
which is a sort of Michael representation (see \cite{rw98}) of $S$ consisting of \cadlag\ selections. In this section we analyze this condition in terms of standard continuity properties of $S_t$ as a function of $t$. These results are of independent interest in set-valued analysis. Our results also lead to sufficient conditions for the predictability of $\tilde S$ in Theorem~\ref{thm:cifcpcor}.

Recall that a function is right-continuous ($\cad$) in the usual sense if and only if it is continuous with respect to the topology $\tau_r$ generated by the right-open intervals $\{[s, t) \mid s<t\}$. A set-valued mapping $\Gamma:[0,T]\tos\reals^d$ is said to be {\em right-inner semicontinuous} (right-isc) if  $\Gamma^{-1}(O)$ is $\tau_r$-open for any open $O\subseteq\reals^d$. Left-inner semicontinuity (left-isc) is defined analogously using the topology $\tau_l$  generated by the left-open intervals $\{(s, t] \mid s<t\}$.

The mapping $\Gamma$ is said to be {\em right-outer semicontinuous} (right-osc) if  its graph is closed in the product topology of $\tau_r$ and the  usual topology on $\reals^d$. The mapping $\Gamma$ is {\em right-continuous} ($\cad$) if it is both right-isc and right-osc. Left-outer semicontinuous (left-osc) and left-continuous ($\cag$) mappings are defined analogously.
We say that $\Gamma$ {\em has limits from the left} ($\lag$) if, for all $t$,
\[
\liminf_{s\upuparrows t} \Gamma_s=\limsup_{s\upuparrows t} \Gamma_s,
\]
where the limits are in the sense of \cite[Section 5.B]{rw98} and are taken along strictly increasing sequences. {\em Having limits from the right} ($\lad$) is defined analogously. A mapping $\Gamma$ is {\em \cadlag\ } (resp. {\em \caglad\ }) if it is both $\cad$ and $\lag$ (both $\cag$ and $\lad$). Recall that a convex-valued $\Gamma$ is {\em solid} if $\interior \Gamma_t\ne\emptyset$ for all $t$.  For any mapping $\Gamma$ we let $\vec \Gamma_{0}:= \{0\}$ and for $t>0$
\begin{equation*}
\vec\Gamma_t:= \liminf_{s \upuparrows t} \Gamma_s.
\end{equation*}

In the following theorem, the distance of $x$ to $\Gamma_t$ is defined, as usual, by  
\[
d(x, \Gamma_t) =\inf_{ x' \in \Gamma_t} d(x,x'),
\]
where the distance of two points is given by the  euclidean metric.

\begin{theorem}\label{thm:cadlagGamma}
Let $\Gamma:[0,T]  \rightrightarrows \reals^d$ be  a \cadlag\ nonempty convex-valued mapping. For every $x\in\reals^d$, the function $y$ defined by
\[
y_t=\argmin_{x'\in\Gamma_t} d(x,x')
\]
satisfies $y\in D(\Gamma)$ and
\begin{equation*}
y_{t-}=\argmin_{x'\in\vec\Gamma_t} d(x,x').
\end{equation*}
In particular,
\[
\Gamma_t=\cl \{y_t \mid y \in D(\Gamma)\}
\]
and $\vec \Gamma$ is \caglad\ nonempty convex-valued with
\[
\vec \Gamma_t=\cl \{y_{t-}\mid y\in D(\Gamma)\}.
\]
\end{theorem}
\begin{proof}
By strict convexity of the distance mapping, the argmin in the definition of $y$ is single-valued. By \cite[Proposition 4.9]{rw98}, $y$ is $\cad$. On the other hand, for every strictly increasing $t^\nu\nearrow t$, $\Gamma_{t^\nu}\to \vec \Gamma_t$, so $y$ is $\lad$, by \cite[Proposition 4.9]{rw98} again. 

Next we show
\begin{equation*}
y_{t-}=\argmin_{x'\in\vec\Gamma_t} d(x,x').
\end{equation*}
 Since $\Gamma$ is \lag, we get $y_-\in D_l(\vec\Gamma)$, so the inequality $d(x, \vec \Gamma_{\bar t})\le d(x, y_{t-}) $ is trivial. For the other direction, assume for a contradiction that  $d(x, \vec \Gamma_{\bar t})<d(x, y_{\bar t-}) $ for some $\bar t \in (0,T]$. There is $s<\bar t$ such that  $d(x, \vec \Gamma_{\bar t})<d(x, y_{s'}) $ for all $s'\in(s,\bar t)$. By the definition of $\vec\Gamma$, this means that 
\[
y_{s'}\notin\argmin_{x'\in\Gamma_{s'}} d(x,x')
\]
for some $s'\in(s,\bar t)$, which is a contradiction. 

The claims $\Gamma_t=\cl \{y_t \mid y \in D(\Gamma)\}$ and $\vec \Gamma_t=\cl \{y_{t-}\mid y\in D(\Gamma)\}$ are now immediate while $\vec\Gamma$ is \caglad\ due to \cite[Exercise 4.2]{rw98}.
\end{proof}

\begin{theorem}\label{thm:solidGamma}
Let $\Gamma:[0,T]  \rightrightarrows \reals^d$ be a closed convex-valued solid mapping such that $\vec \Gamma$ is solid. We have
\begin{equation*}
\Gamma_t =\cl \{y_t \mid y \in D(\Gamma)\}
\end{equation*}
if and only if $\Gamma$ is  right-isc. In this case, $\vec \Gamma$ is left-isc and
\[
\vec \Gamma_t =\cl \{ y_{t-} \mid y \in D(\Gamma)\}.
\]
\end{theorem}
\begin{proof}
Necessity is obvious. To prove sufficiency we first show that $D(\Gamma)\ne\emptyset$.

For $\bar t \in [0,T)$ and $y^r  \in \interior \Gamma_{\bar t}$, there exists $\delta>0$ such that $y^r \in \Gamma_u$ for $u \in [\bar t,\bar t+\delta]$ since $\Gamma$ is right-isc and solid. Indeed, let $\bar y^i$ be a finite set of points in $\interior \Gamma_{\bar t}$ such that $y^r$ belongs to the interior of the convex hull $\co\{\bar y^i\}$. Let $\epsilon>0$ be small enough so that $y^r \in \co\{ v^i\}$ whenever, for every $i$, $v^i\in (\bar y^i+\epsilon \uball)$. Since $\Gamma$ is right-isc, there is, for every $i$, $u^i>\bar t$ such that $\Gamma_u\cap(\bar y^i+\epsilon \uball)\ne\emptyset$ for every $u\in[\bar t,u^i)$. Denoting $\bar u=\min u^i$, we have, by convexity of $\Gamma$, that 
\begin{equation}\label{eq:kp}
y^r\in \Gamma_u \mbox{ for every } u\in[\bar t,\bar u).
\end{equation}
Now assume $\bar t>0$ and  take $y^l  \in \interior \vec \Gamma_{\bar t}$.  We now show the existence of $s<\bar t$ such that
\begin{equation}\label{eq:kpleft}
y^l \in \Gamma_u \mbox{ for every } u\in(s,\bar t].
\end{equation}
Assume for a contradiction the existence of $t^\nu \nearrow \bar t$ such that $y^l \notin \Gamma_{t^\nu}$.  Let $\bar y^i \in \interior \vec \Gamma_{\bar t}$ be $d+1$ points and $\epsilon>0$  such that for any points $\tilde y^i \in \bar y^i  + \epsilon \uball$, $\tilde y^i \in  \vec \Gamma_{\bar t}$ and $y^l \in \co\{\tilde y^i \}$.  By the definition of $\vec \Gamma$ as a left-limit, there exists $\nu_0 \in \naturals$ such that for all $\nu> \nu_0$ there exists $y^i_{\nu} \in \Gamma_{t^\nu}$ with $y^i_{\nu} \in \bar y^i  + \epsilon \uball$ for $i=1,\ldots,d+1$. Then,  $y^{l} \in \co\{y^i_{\nu}\}$ and this last set is included in $\Gamma_{t^\nu}$ by convexity. Then, $y^{l} \in \Gamma_{t^\nu}$, a contradiction. 

We have proved that for any $t \in [0,T]$ there exists $\delta >0$ and  $y^l,y^r \in \reals^d$ such that $y_u = y^l 1_{[ (t-\delta)^+,  t)}(u) + y^r 1_{[t,t+\delta]}(u)$ is  a $\cadlag$ selection of $\Gamma$ in the interval $[t-\delta,t+\delta] \cap [0,T]$.  Now we can paste together these local selections using partitions of unity as in \cite[Theorem 3.2'']{mic56}. We have proved that  $D(\Gamma) \neq \emptyset$. Now $\Gamma_t =\cl \{y_t \mid y \in D(\Gamma)\}$ is an easy consequence to \eqref{eq:kp} and $D(\Gamma) \neq \emptyset$.  We have shown the sufficiency.

To prove $\vec \Gamma_t =\cl \{ y_{t-} \mid y \in D(\Gamma)\}$, the inclusion $\supseteq$ is clear. Now take $y^l \in \interior \vec \Gamma_{\bar t}$ and $s<\bar t$ as in  \eqref{eq:kpleft} and $y \in D(\Gamma)$. Defining
\[
\begin{cases}
y^l \quad & t \in [s, \bar t)\\
y_t \quad & otherwise,
\end{cases}
\]
we get the inclusion $\subseteq$. From the representation $\vec \Gamma_t =\cl \{ y_{t-} \mid y \in D(\Gamma)\}$ it follows that $\vec \Gamma$ is left-isc.
\end{proof}

The following preparatory lemma characterizes preimages of $\vec\Gamma$. For a set $A$, $\vec\cl A$ denotes the limit points from the left of $A$. Thus, for $t \in \vec \cl  A$ and $s<t$,  $(s,t) \cap A \neq \emptyset$.
\begin{lemma}\label{lem:vecCl}
For any closed set $A\subset\reals^d$, 
\[
\vec\Gamma^{-1}(A)= \bigcup_{N \in \naturals} \bigcap_{\nu \in \naturals} \bigcap_{n \in \naturals}  \bigcup_{a \in \rationals^d \cap N \uball} \left[ \vec\cl C_{\nu,n,a}   / \vec\cl C_{\nu,n,a}^c\right],
\]
where $C_{\nu,n,a}:=\Gamma^{-1}(A\cap (a+n^{-1}\uball)+\nu^{-1}\uball)$.
\end{lemma}
\begin{proof}
To prove the inclusion $\supseteq$, let 
\[
t \in \bigcap_{\nu \in \naturals} \bigcap_{n \in \naturals}  \bigcup_{a \in \rationals^d \cap N \uball} \left[ \vec\cl C_{\nu,n,a}   / \vec\cl C_{\nu,n,a}^c\right]
\]
 for some fixed $N$. For all $\nu,n$ there exists $a^{\nu,n} \in  \rationals^d \cap N \uball$ such that $t \in \vec\cl C_{\nu,n,a^{\nu,n}}   / \vec\cl C_{\nu,n,a^{\nu,n}}^c$.  By compactness, there exists a subsequence $n_j \to \infty$ and $a^* \in \reals^d$ such that $a_j:=a^{n_j,n_j}\to a^*$.  There exists $s_j<t$ such that $(s_j,t] \subset C_{n_j,n_j,a_j}$ since $t \in \vec\cl C_{n_j,n_j,a_j}   / \vec\cl C_{n_j,n_j,a_j}^c$. As a consequence, for any sequence $t_k \nearrow t$, there exists $y_k \in  \Gamma_{t_k}$ with $y_k \in A \cap (a_{j_k}+n^{-1}_{j_k}\uball) + n^{-1}_{j_k}\uball$ for an appropriate subsequence $j_k \to \infty$. Then $y_k \to a^*$ and $a^* \in \bar A=A$.  Thus $t \in \vec \Gamma^{-1}(A)$.

Now we show $\subseteq$. Take  $t>0$ in  $\vec \Gamma^{-1}(A)$ and $a^* \in A \cap \vec \Gamma_t \cap N \uball$.  Fix $n$ and $\nu$. There exists $a \in \rationals^d \cap N\uball$ such that $a^* \in A \cap (a+n^{-1}\uball)$. Then $t \in  \vec\Gamma^{-1}(A \cap (a+n^{-1}\uball))$ and this yields the existence of $s<t$ such that $(s,t) \subset C_{\nu,n,a}$. Thus $t \in \vec\cl C_{\nu,n,a}   / \vec\cl C_{\nu,n,a}^c$.
\end{proof}






We finish the section with the stochastic setting. For $\Gamma:\Omega\times[0,T]\tos\reals^d$, $\vec\Gamma$ is defined pathwise as above, and for $C\subset \Omega\times [0,T]$, $\vec\cl C$ denotes the pathwise limit points from the left. 

\begin{lemma}\label{lem:variantIV32}
Let $A$ be a $\F \otimes \B([0,T])$-measurable set. Then $\vec \cl A$ is measurable.
\end{lemma}
\begin{proof}
The process $D_s(\omega)=\inf \{t>s \mid (t,\omega) \in A \}$ is measurable; see the proof of \cite[Theorem IV.32]{dm78}. Since
\[
\vec \cl A = \bigcap_{\nu \in  \naturals} \{ (s,\omega) \mid D_{(s - \frac{1}{\nu})^+}(\omega) <s \},
\]
 $\vec \cl A$ is thus measurable.
\end{proof}

\begin{lemma}\label{lem:vecGpredictability}
Let $\Gamma:\Omega\times[0,T]\tos\reals^d$ be a measurable mapping. Then $\vec\Gamma$ is measurable.  If $\Gamma$ is progressively measurable, then $\vec \Gamma$ is predictable.
\end{lemma}
\begin{proof}
For a closed set $A \subset \reals^d$,  Lemma~\ref{lem:vecCl} gives 
\[
\vec\Gamma^{-1}(A)= \bigcup_{N \in \naturals} \bigcap_{\nu \in \naturals} \bigcap_{n \in \naturals}  \bigcup_{a \in \rationals^d \cap N \uball} \left[ \vec\cl C_{\nu,n,a}   / \vec\cl C_{\nu,n,a}^c\right],
\]
where the sets on the right hand side are measurable by Lemma \ref{lem:variantIV32}. Hence, $\vec \Gamma^{-1}(A)$ is a measurable set. Then, $\vec \Gamma$ is a measurable mapping  by \cite[Theorem 14.3 (b)]{rw98}. When $\Gamma$ is progressively measurable, the sets on the right are predictable by \cite[Theorem IV.89]{dm78}.
%
\end{proof}

Combining results of this section, we get sufficient conditions for the predictability assumption in Theorem~\ref{thm:cifcpcor}.
\begin{theorem}\label{thm:vectilde}
Let $h$ be an optional normal integrand and $S_t(\omega)=\cl\dom h_t(\omega)$. Then the set-valued mapping defined by
\[
\tilde S_t(\omega)=\cl\{y_{t-}\mid y\in D(S)(\omega)\}
\] 
is predictable and coincides with $\vec S$ under either of the following conditions:
\begin{enumerate}
\item $S$ is $\cadlag$,
\item $S$ and $\vec S$ are solid.
\end{enumerate}
\end{theorem}
\begin{proof}
By Theorems~\ref{thm:cadlagGamma} and \ref{thm:solidGamma}, $\tilde S=\vec S$ under either condition, so the claim follows from Lemma~\ref{lem:vecGpredictability}.
\end{proof}


\section{Applications}\label{sec:applications}
In this section we demonstrate how Corollary~\ref{cor:S} together with Theorem~\ref{thm:vectilde} lead to well-known models in mathematical finance. The article \cite{pp18c} applies our results to optimal stopping while partial hedging of American options will be studied in a forthcoming article by the authors. Further applications to finance and to singular stochastic control  will be presented elsewhere. We assume throughout the section the conditions of Theorem \ref{thm:cifcpcor} for $\D$. That is, $\D$, is a solid max-stable space of adapted $\cadlag$ processes of class $(D)$ containing $\D^\infty$. 

 For a stochastic process $b$, 
\[
\vec b_{t}:=\limsup_{s\upuparrows t} b_s
\]
is its {\em left upper semicontinuous regularization}. A process is {\em left-usc} if $b \ge \vec{b}$. By \cite[Theorem~IV.90]{dm78}, $\vec b$ is predictable whenever $b$ is optional. For a \cadlag\ $b$, $\vec b$ is \caglad\ and $\vec b=b_-$. Likewise, $b$ is  said to be {\em right-usc} if $b\ge \cev b$, where 
\[
\cev b_{t}:=\limsup_{s\downdownarrows t} b_s
\]
is the {\em right-upper semicontinuous regularization} of $b$.  These regularizations appear in the context of optimal stopping as properties of reward processes, e.g., in \cite{Bismutskallitempsarret,elk81}, or more recently in and \cite{bb18a,pp18c}.

\begin{example}\label{ex:1}
Let $b$ be optional and right-usc dominated by some $\check y\in\D$ and 
\begin{align*}
S_t(\omega) &:=\{y\in\reals\mid y\ge b_t(\omega)\}.
\end{align*} 
Then $\D(S)$ is nonempty closed convex with
\[
\sigma_{\D(S)}(u,\tilde u)= E\left[\int b du+\int \vec b d\tilde u\right]+\delta_{\hat \M_-}(u,\tilde u).
\]
Moreover, $\sigma_{\D(S)}$ is the lower semicontinuous hull of
\[
(u,\tilde u)\mapsto E\int b du+\delta_{\M_-\times\{0\}}(u,\tilde u).
\]
\end{example}
\begin{proof}
Let
\[
\tilde S_t(\omega) :=\{y\in\reals\mid y\ge \vec b_t(\omega)\}.
\]	
The result follows from Corollary~\ref{cor:S} once we verify that $\tilde S_t =\cl\{ y_{t-}\mid y\in D(S)\}$. In the following argument, $\omega$ is fixed.

	
If $y\in D(S)$, then  $ y_- =\vec y \ge \vec b$, so $\tilde S_t \supseteq \cl\{ y_{t-}\mid y\in D(S)\}$. To prove the converse,  it suffices to show that $\bar y:=\vec b_{\bar t} \in \cl\{ y_{\bar t-}\mid y\in D(S)\}$ for a given $\bar t\in (0,T]$. Fix $\epsilon>0$. Since $\vec b$ is left-usc, the set $\{s\mid \vec b_s < \bar y +\epsilon\}$ contains  a $\tau_l$-neighborhood of $\bar t$. Thus $b_s < \bar y + \epsilon$ for all $s \in [u, \bar t)$ for some $u < \bar t$.   We have $z \ge b$ and $z_{\bar t-}= \vec b_{\bar t}+\epsilon$, where 
\[
z_t := (\bar y + \epsilon)\one_{[u,\bar t)}(t) + \check y_t\one_{[u,\bar t)^C}(t).
\]
Since $\epsilon>0$ was arbitrary, $\vec b_{\bar t}\in \cl\{y_{t-}\mid y\in D(S)\}$.
\end{proof}

Let $a$ be a stochastic process and $\undervec a_{t}:=\liminf_{s\upuparrows t} a_s$.  By construction, $\undervec a$ is left-lsc. By \cite[Theorem~IV.90]{dm78}, $\undervec a$ is predictable whenever $a$ is predictable. In the next example, the set-valued mapping $S$  describes {\em bid-ask spreads} in proportional transaction cost models; see \cite{cpsy18} and references therein. Our example allows for general bid and ask prices given by right-usc and right-lsc processes. 

\begin{example}\label{ex:2}[Bid-ask spreads]
Let $b$ be optional and right-usc and $a$ be optional and right-lsc such that there exists $\bar y\in\D$ with $b<\bar y<a$ and $\vec b < \bar y_{-}  < \undervec a$ a.s.e. Let
\begin{align*}
S_t(\omega)&:=\{y\in\reals\mid b_t(\omega)\le y \le a_t(\omega)\}
\end{align*} 
Then $\D(S)$ is nonempty closed convex with 
\[
\sigma_{\D(S)}(u,\tilde u)= E\left[\int a du^+-\int b du^-+\int \vec a d\tilde u^- - \int \undervec bd\tilde u^-\right]. 
\]
Moreover, $\sigma_{\D(S)}$ is the lower semicontinuous hull of
\[
(u,\tilde u)\mapsto E\left[\int a du^+ - \int b du^-\right]+\delta_{\{0\}}(\tilde u).
\]
\end{example}
\begin{proof}
Let
\[
\tilde S_t(\omega):=\{y\in\reals\mid \vec b_t(\omega)\le y \le \undervec a_t(\omega)\}.
\]
The result follows from Corollary~\ref{cor:S} provided that $\tilde S_t=\cl\{y_{t-}\mid y\in D(S)\}$. To this end,  one may proceed as in Example~\ref{ex:1} using $\bar y$ as a bound.
%
\end{proof}
 
In the next example, $\C$ describes the set of {\em self-financing portfolio processes} in a currency market of $d$ different currencies, and the martingales $y\in \D$ with $y\in\D(S)$ and $y_-\in\D_l(\tilde S)$ are {\em consistent price systems}. We refer to \cite{ks9} for a detailed treatment of the subject. The claims in the example follow immediately from Theorem~\ref{thm:vectilde} and Corollary~\ref{cor:S}.
\begin{example}[Currency markets]
Assume that $S$ is an optional right-isc solid convex cone-valued mapping and that $\vec S$ is solid. Then $\D(S)$ is a closed convex cone and its polar cone is
\[
\C=\{(u,\tilde u)\in \hat \M\mid (du/d|u|)_t\in S^*_t, (d\tilde u/d|\tilde u|)_t\in \vec S^*_t\},
\]
where $S_t^*(\omega)=\{x\mid x\cdot y \le 0\ \forall\ y\in S_t(\omega)\}$.
\end{example}
We finish this section by showing that the assumptions in the above currency market model are satisfied by the general model in \cite[Section 3.6.6]{ks9}.
 \begin{example}[Campi-Schachermayer model]
Let $G$ (resp. $\tilde G$) be an optional (resp. predictable) closed convex cone-valued mapping. We assume
\begin{itemize}
\item ``Efficient friction:'' $G_t\cap(-G_t)=\{0\}$ and $\tilde G_t\cap(- \tilde G_t)=\{0\}$ for all $t$;
\item ``Regularity hypotheses'': $G_{t,t+}=G_t$ and $G_{t-,t}=\tilde G_{t}$, where
\begin{align*}
G_{s,t}(\omega) &:=\cl \cone \{G_r(\omega)\mid r\in[s,t)\},\\
G_{s,t+}(\omega)&:=\bigcap_{\epsilon>0} G_{s,t+\epsilon}(\omega),\\
G_{s-,t}(\omega)&:=\bigcap_{\epsilon>0} G_{s-\epsilon,t}(\omega),
\end{align*}
and the last is defined as $G_{0,t}(\omega)$ for $s=0$.
\end{itemize}
 Then $S=G^*$ is right-isc solid-valued and $\vec S=\tilde G^*$ is solid.
\end{example}
\begin{proof}
The regularity assumptions $G_{t,t+}=G_t$ and $G_{t-,t}=\tilde G_{t}$ imply that $G$ is right-osc and that $\tilde G$ is left-osc; we show the latter, the former being simpler. Let $t^\nu\nearrow t$, $y^\nu\in \tilde G_{t^\nu}$ with $y^\nu\to y$. We need to show $y\in \tilde G_t$ to which end we may assume that $t^\nu<t$ for each $\nu$, otherwise the claim is trivial. For any $\epsilon>0$, $t^\nu\in(t-\epsilon,t)$ for $\nu$ large enough. Since $y_{t^\nu}\in \tilde G_{t^\nu}=G_{t^\nu-,t^\nu}$, there are sequences $t_k^\nu\nearrow t^\nu$ and $y_{t^\nu_k}\to y_{t^\nu}$ with $t_k^\nu>t-\epsilon$ and $y_{t^\nu_k}\in G_{t^\nu_k}$. By diagonalization, there exists $k(\nu)$ such that $t^\nu_{k(\nu)}\nearrow t$ and $y_{t^\nu_{k(\nu)}}\to y$. Thus $y\in G_{t-\epsilon,t}$ and, since $\epsilon>0$ was arbitrary, $y\in G_{t-,t}=\tilde G_{t}$.

 By \cite[Corollary 11.35]{rw98}, $G^*$ and $\tilde G^*$ are thus right- and left-isc, respectively while $G_{t-,t}=\tilde G_{t}$ then means that $\tilde G^*_t =\vec G^*_t$. By \cite[Corollary 13.4.2]{roc70a}, a convex cone $K$ is solid if and only if $K\cap(-K)=\{0\}$, so the efficient friction assumption implies that $G^*$ and $\tilde G^*$ are solid-valued.
\end{proof}

\section{Appendix}
The following extends \cite[Lemma~4.2]{pp18a} that was formulated for bounded $w$.

\begin{lemma}[Jensen's inequality]\label{lem:jin}
Assume that $h$ is an optional convex normal integrand, $\mu$ is an optional nonnegative  random measure,
\[
h(x)\ge x\cdot v-\alpha
\] 
for some optional $v$ and nonnegative $\alpha$ such that $\int |v|d\mu$ and $\int\alpha d\mu$ are integrable, and that $w$ is a raw  measurable process with $E\int |w||v|d\mu<\infty$. If $w$ has an optional projection, then
\[
EI_h(w)\ge EI_h(\op w).
\]
If $h$, $\mu$ and $v$ are predictable and $w$ has a predictable projection, then
\[
EI_h(w)\ge EI_h(\pp w).
\]
\end{lemma}
\begin{proof}
Let $\hat\mu\ll \mu$ be defined by $d\hat\mu/d\mu=\beta:=\op(1/(1+\int d\mu))$. Then  $\hat \eta(A)=E\int\one_A d\hat\mu$ defines  an optional bounded measure  on $\Omega\times[0,T]$. Moreover, $EI_h(w)=E\int \hat h(w)d\hat\mu$, where $\hat h(w)=h(w)/\beta$ is an optional convex normal integrand. We have
\[
\hat h^*(v)=h^*(\beta v)/\beta,
\]
so the lower bound implies that $E\int \hat h^*(v/\beta)d\hat\mu$ is finite. Thus we may apply the interchange of integration and minimization on $(\Omega\times[0,T],\O,\hat\eta)$ and on $(\Omega\times[0,T],\F\otimes\B([0,T]),\hat\eta)$ (see \cite[Theorem 14.60]{rw98}) to get
\begin{align*}
EI_h(\op w) &=E\int \hat h(\op w)d\hat\mu\\
 &=\sup_{v\in \LL^1(\Omega\times[0,T],\O,\hat\eta)}E\int[\op w\cdot v-\hat h^*(v)]d\hat\mu\\
&=\sup_{v\in \LL^1(\Omega\times[0,T],\O,\hat\eta)}E\int[w\cdot v-\hat h^*(v)]d\hat\mu\\
&\le \sup_{v\in \LL^1(\Omega\times[0,T],\F\otimes\B([0,T]),\hat\eta)}E\int[w\cdot v-\hat h^*(v)]d\hat\mu\\
&=E\int \hat h(w)d\hat\mu\\
&=EI_h(w).
\end{align*}
The predictable case is proved similarly.
\end{proof}

The next lemma was used in the proof of Theorem \ref{thm:ircp}.

\begin{lemma}\label{lablemmaAlternativeannouncing}
Let $\tau$ be a predictable time such that $[ \tau ]  \subset ( m,m+1 ]$ for a given $m \in \{0,1,\ldots,\}$. Let $\{\epsilon_n\}_{n \in  \nn}$ be a decreasing sequence of positive real numbers converging  to zero.  Then, there exists a nondecreasing sequence of  stopping times $\{\sigma^n\}_{n \in \nn}$ converging to $\tau$ with  $[ \sigma^n ] \subset ( m,m+1 ) $ for every $n$, and with the following properties: $\{\tau <\infty\} \subset \{\sigma^n < \tau\}$   and
\begin{equation}\label{labEqFastlabelingAnnSeqInfiniteonInfinite}
\left| \rpr(\{\sigma^{n} = \infty\}) -\rpr(\{\tau=\infty\})\right|  \leq \epsilon_{n}.
\end{equation}
\end{lemma}
\begin{proof}
By \cite[Theorem IV.11]{Dellacherie1972}, there exists an announcing sequence $\{ \tau^n \}_{n \in \nn}$ of $\tau$ with graphs on $\llparenthesis m,m+1 \rrparenthesis$ and  $\{\tau <\infty\}\subset\{\tau^n < \tau\}$. Note that 
\[
\{\tau=\infty\}= \bigcup_{n \in \nn} \{\tau^n  >m+1 \}.
\]
Then, for any  sequence $\{\epsilon_n\}_{n \in \nn}$ decreasing to zero, we might assume, by taking a subsequence if necessary, that
\begin{equation}\label{labEqFastlabelingAnnSeq}
\left| \rpr(\{\tau^{n} > m+1\})-\rpr(\{\tau=\infty\})\right|  < \epsilon_{n}.
\end{equation}
For $n \in \nn$,  let $\sigma^n$ be the  stopping time $\tau^n_{\{\tau^n   \leq  m+1 \}}$. The sequence $\{\sigma^n \}_{n\in \nn}$ is non-decreasing and converges to $\tau$. Moreover
\[
\rpr(\sigma^n=\infty)=\rpr(\tau^n > m+1)   \geq \rpr(\tau = \infty)- \epsilon_n,
\]
which proves the claim.
\end{proof}
\textbf{Acknowledgements} Erick Trevi\~no gratefully acknowledges financial support from Alexander von Humboldt Foundation  while visiting the Technical University of Berlin.

\bibliographystyle{plain}
\bibliography{sp}

\end{document}